\documentclass{amsart}

\usepackage[left=2.5cm, right=2.5cm, top=2cm, bottom=2cm]{geometry}
\usepackage{amsmath}
\usepackage{amssymb}
\usepackage{amsthm}
\usepackage{physics}
\usepackage{tikz}
\usepackage{dcolumn}
\usepackage{placeins}
\usepackage{mathtools}
\usepackage{tabu}
\usepackage{colortbl}
\usepackage[hidelinks]{hyperref}

\usetikzlibrary{arrows.meta}

\newtheorem{theorem}{Theorem}

\newmuskip\pFqmuskip
\newcommand*\pFq[6][8]{%
  \begingroup
  \pFqmuskip=#1mu\relax
  \mathcode`\,=\string"8000
  \begingroup\lccode`\~=`\,
  \lowercase{\endgroup\let~}\pFqcomma
  {}_{#2}F_{#3}{\left(\genfrac..{0pt}{}{#4}{#5};#6\right)}%
  \endgroup
}
\newcommand{\pFqcomma}{\mskip\pFqmuskip}

\makeatletter
\renewcommand*\env@matrix[1][\arraystretch]{%
  \edef\arraystretch{#1}%
  \hskip -\arraycolsep
  \let\@ifnextchar\new@ifnextchar
  \array{*\c@MaxMatrixCols c}}
\makeatother

\begin{document}

\author{Kilian B\"onisch}
\email{boenisch@physik.uni-bonn.de}

\author{Vasily Golyshev}
\address{PRIMI, Polo di Ricerca Interuniversitario in Matematica e Informatica,  Italy}
\email{vasily.v.golyshev@gmail.com}

\author{Albrecht Klemm}
\address{Department of Mathematical and Physical Sciences, University of Sheffield \newline S3 7RH Sheffield, United Kingdom}
\email{a.klemm@sheffield.ac.uk}

\title{Fibering out Calabi-Yau motives}

\begin{abstract}
    We prove the modularity of mixed periods associated with singular fibers of specific families of Calabi-Yau threefolds. This is done by ``fibering out'', i.e.\ by expressing these periods as integrals of periods of families of K3 surfaces and by using modularity properties of the latter. Besides classical periods of holomorphic modular forms and meromorphic modular forms with vanishing residues, the computations lead to new interesting periods associated with meromorphic modular forms with non-vanishing residues as well as contours between CM points. 
\end{abstract}

\maketitle

\section{Introduction}

In the main part of this paper we illustrate the method of ``fibering out'' with the famous family of hypersurfaces $X_\psi$ given by the vanishing set
\begin{align*}
    x_1^5 + x_2^5 + x_3^5 + x_4^5 + x_5^5 -5 \, \psi \, x_1 \, x_2 \, x_3 \, x_4 \, x_5 \, = \, 0
\end{align*}
in $\mathbb{P}^4$. The so-called conifold fiber $X_1$ is singular and in \cite{Schoen} it has been shown that the Galois representations on the middle cohomology of a resolution $\widehat{X_1}$ of $X_1$ are modular. More concretely, for all primes $p \neq 5$ and any prime $\ell \neq p$ one has
\begin{align*}
    \det(1-\text{Frob}_p \, T \, | \, H^3(\overline{\widehat{X_1}},\mathbb{Q}_{\ell})) \, = \, 1-a_p \, T + p^3 \, T^2
\end{align*}
with the Hecke eigenvalues $a_p$ of the unique newform $f \in S_4(\Gamma_0(25))$ with Hecke eigenvalue $a_2 = 1$. This suggests that the period matrix of $\widehat{X_1}$ should be given by the period matrix of $f$, and this has been numerically verified in \cite{ADEK}. In the following sections we prove this and extend the result to a mixed period matrix of rank four associated with the limit $\psi \rightarrow 1$. To do this, we follow steps outlined in the PhD thesis of one of the authors \cite{KilianThesis}, which are based on the idea of ``fibering out'' from \cite{GolyshevZagier}. In Section \ref{sec:ConifoldPeriods} we review the structure of the mixed period matrix $T$ of rank four associated with the limit $\psi \rightarrow 1$. In Section \ref{sec:FiberingOut} we fiber out periods of $X_\psi$, i.e. we express them as integrals of period functions of a family of K3 surfaces. In Section \ref{sec:Modularity} we finally use modularity properties of the family of K3 surfaces to express all mixed periods associated with the conifold fiber as integrals of modular forms. Qualitatively, the resulting identities for the most interesting part of $T$ have the form
\begin{align*}
  \begin{pmatrix}
    w_+ & e_+ & a_+ \\
    w_- & e_- & a_- \\
    b & d & c \\
  \end{pmatrix}
  \, = \,
  \begin{pmatrix}[1.5]
    \oint_{\gamma_+} \dd \tau \\
    \oint_{\gamma_-} \dd \tau \\
    \int_{\tau_-}^{\tau_+} \dd \tau
  \end{pmatrix}
  \, (f_{50}, F_{50}, g_{50}) \, ,
\end{align*}
where $\gamma_{\pm}$ are paths between cusps, $\tau_\pm$ are CM points, $f_{50}$ is a holomorphic modular form, $F_{50}$ is a meromorphic modular form with vanishing residues and $g_{50}$ is a meromorphic modular form with non-vanishing residues. \\

In Appendix \ref{sec:AppendixOtherFamilies} we list other families of Calabi-Yau threefolds for which the fibering out proves the modularity of the associated mixed period matrix. 

\medskip

\bf Acknowledgements. \rm Conversations with Emmanuel Scheidegger and Don Zagier  have contributed a lot to our understanding of the subject. We thank Charles Doran for his private communication and for pointing us to \cite{DoranIdentities}. We thank the members of the International Groupe de Travail on differential equations in Paris, and specifically Spencer Bloch, Matt Kerr, and Wadim Zudilin, for their insights.

\section{Mixed periods associated with the conifold fiber}
\label{sec:ConifoldPeriods}

As long as $\psi^5 \neq 1$, the hypersurface $X_\psi$ is a smooth Calabi-Yau threefold, i.e.\ a smooth projective variety with trivial canonical bundle and Hodge numbers $h^{1,0} = h^{2,0} = 0$. The Hodge diamond looks as follows:
  \begin{align*}
    \begin{array}{ccccccc}
        &   &     & 1 &     &   &   \\
        &   & 0   &   & 0   &   &   \\
        & 0 &     & 1 &     & 0 &   \\
      1 &   & 101 &   & 101 &   & 1 \\
        & 0 &     & 1 &     & 0 &   \\
        &   & 0   &   & 0   &   &   \\
        &   &     & 1 &     &   &   \\
    \end{array}
  \end{align*}
 This family was studied in the famous paper \cite{MirrorSymmetry} and led to the discovery of mirror symmetry. Of particular importance in this context is the group
\begin{align*}
    G \, = \, \{ (\alpha_1,\alpha_2,\alpha_3,\alpha_4,\alpha_5) \, \in \mu_5^5 \ | \ \alpha_1 \, \alpha_2 \, \alpha_3 \, \alpha_4 \, \alpha_5 \, = \, 1  \} \, ,
\end{align*}
where $\mu_5$ denotes the set of fifths roots of unity. This acts on the fibers of the family by
\begin{align*}
    (x_1 : x_2 : x_3 : x_4 : x_5) \, \mapsto \, (\alpha_1 \, x_1 : \alpha_2 \, x_2 : \alpha_3 \, x_3 : \alpha_4 \, x_4 : \alpha_5 \, x_5) \, .
\end{align*}
The most interesting part of the middle cohomology of the smooth fibers is the four-dimensional part 
\begin{align*}
    V_\psi = H^3(X_\psi)^G
\end{align*}
that is invariant under the action of $G$. Equivalently, this four-dimensional part arises as the complete middle cohomology of a resolution of the quotient $X_\psi/G$, the so-called mirror of $X_\psi$. The Hodge numbers of $V_\psi$ are~$1 \ 1 \ 1 \ 1$ and we are mainly interested in the mixed period matrix associated with the limit $\psi \rightarrow 1$. To give a concrete description of this period matrix we define the holomorphic three-form
\begin{align*}
    \Omega_\psi \, = \, -5 \, \psi \, \text{Res} \,  \frac{\sum\limits_{i=1}^5 (-1)^{i+1} \, x_i \, \dd x_1 \wedge \cdots \wedge \widehat{\dd x_i} \wedge \cdots \wedge \dd x_5}{x_1^5 + x_2^5 + x_3^5 + x_4^5 + x_5^5 - 5 \, \psi \, \, x_1 \, x_2 \, x_3 \, x_4 \, x_5} \, ,
\end{align*}
where $\widehat{\dd x_i}$ denotes the omission of the differential $\dd x_i$. For $|\psi| > 1$ we can integrate $\Omega_\psi$ over a suitable three-dimensional torus to obtain the period\footnote{To derive this integral expression, one first writes the integral of the residue over the three-dimensional torus in $X_\psi$ as an integral over a four-dimensional torus in the complement $\mathbb{P}^4 \smallsetminus X_\psi$. Using the homogeneity of the integrand, one can then express the latter as an integral over a five-dimensional torus.}
\begin{align*}
  \int_{S_1^3} \Omega_\psi \,
  &= \, \frac{1}{(2\pi i)^2} \oint \frac{\dd x_1}{x_1} \cdots \oint \frac{\dd x_5}{x_5} \frac{1}{1-\frac{1}{5 \, \psi}\frac{x_1^5 + \cdots + x_5^5}{x_1 \cdots x_5}} \\
  &= \,  (2\pi i)^3 \sum_{n=0}^\infty \frac{(5 \, n)!}{n!^5} \left( \frac{1}{5 \, \psi} \right)^{5 \, n} \\
  & = \, (2\pi i)^3 \, \pFq{4}{3}{\frac{1}{5},\frac{2}{5},\frac{3}{5},\frac{4}{5}}{1,1,1}{1/\psi^5} \, .
\end{align*}
To obtain periods associated with other three-cycles we can use the monodromy of this period. This is particularly well understood since, in terms of the variable $z = 1 / (5 \, \psi)^5$, the power series above is annihilated by the hypergeometric Picard-Fuchs operator
\begin{align*}
    \mathcal{L} \, = \, \Theta^4-5^5 \, z \, (\Theta+1/5) \, (\Theta+2/5) \, (\Theta+3/5) \, (\Theta+4/5) \qquad \text{with} \qquad \Theta \, = \, z \dv{}{z} \, .
\end{align*}
For $0 < z < 1/5^5$, a basis of solutions of $\mathcal{L}$ is given by
\begin{align*}
    \varpi(z) \, = \,
        \left(
        \begin{array}{r}
            f_1(z) \\ [1 mm]
            \log(z) \, f_1(z) + f_2(z) \\ [1 mm]
            \frac{1}{2} \, \log(z)^2 \, f_1(z) + \log(z) \, f_2(z) + f_3(z) \\ [1 mm]
            \frac{1}{6} \, \log(z)^3 \, f_1(z) + \frac{1}{2} \, \log(z)^2 \, f_2(z) + \log(z) \, f_3(z) + f_4(z) 
        \end{array}
        \right)
 \end{align*}
with convergent power series normalized by $f_1(0) = 1$ and $f_2(0)=f_3(0)=f_4(0)=0$. In terms of this basis, the period over the three-dimensional torus corresponds to $(2\pi i)^3 \, \varpi_1$ and a basis of period functions is given by
\begin{align*}
\renewcommand{\arraystretch}{1.2}
    \Pi \, = \,
    \left(
    \begin{array}{c c c c}
      (2\pi i)^3 & 0 & 0 & 0 \\
      0 & (2\pi i)^2 & 0 & 0 \\
      50 \, \frac{(2\pi i)^3}{24} &  \frac{1}{2} \, (2 \pi i)^2 & -5 \, (2\pi i) & 0 \\
      -200 \, \zeta(3) & 50 \, \frac{(2\pi i)^2}{24} & 0 & 5
    \end{array}
    \right)
    \, \varpi \, .
\end{align*}
In this basis, the monodromy matrices (acting by $\Pi \mapsto M \Pi$) are integral and symplectic with respect to the intersection matrix
\begin{align*}
  \Sigma \, = \,
  \left(
  \begin{array}{cccc}
    0 & 0 & 0 & 1 \\
    0 & 0 & 1 & 0 \\
    0 & -1 & 0 & 0 \\
    -1 & 0 & 0 & 0 \\
  \end{array}
  \right)
  \, .
\end{align*}
For loops counterclockwise around $z=0$ and $z=1/5^5$, the monodromy matrices are given by
\begin{align*}
  \renewcommand{\arraystretch}{1.2}
  M_0 \, = \,
  \left(
  \begin{array}{c c c c}
    1 & 0 & 0 & 0 \\
    1 & 1 & 0 & 0 \\
    -2  & -5  & 1 & 0 \\
     5 & 3  & -1 & 1 \\
  \end{array}
  \right)
  \qquad \text{and} \qquad
  M_{1/5^5} \, = \,
  \left(
  \begin{array}{c c c c}
    1 & 0 & 0 & -1 \\
    0 & 1 & 0 & 0 \\
    0 & 0 & 1 & 0 \\
    0 & 0 & 0 & 1 
  \end{array}
  \right) \, .
\end{align*}
Finally, the mixed period matrix $T$ associated with the limit $\delta = 1-5^5 \, z \rightarrow 0$ is defined by \begin{align*}
  \Pi(z) \, = \, T \,
  \begin{pmatrix}
    \log(\delta) \, \nu(\delta)+O(\delta^3) \\
    1+O(\delta^3) \\
    \delta^2+O(\delta^3) \\
    \nu(\delta)
  \end{pmatrix}
\end{align*}
with the so-called vanishing period function $\nu(\delta) = \delta+O(\delta^2)$. Using the monodromy matrices, the intersection pairing and that $f_1(z)>0$ for $0<z<1/5^5$, one finds that
\begin{align}
  T \, = \,
  \begin{pmatrix}
    -2\pi i \, \sqrt{5} & b & d & c \\
    0 & w_+ & e_+ & a_+ \\
    0 & \frac{1}{2} \, w_++w_- & \frac{1}{2} \, e_++e_- & \frac{1}{2} \, a_++a_- \\
    0 & 0 & 0 & (2\pi i)^2 \, \sqrt{5}
  \end{pmatrix}
  \label{eq:MixedPeriods}
\end{align}
with real constants $w_+$, $e_+$, $a_+$ and purely imaginary constants $w_-$, $e_-$, $a_-$, $b$, $d$, $c$ satisfying
\begin{align*}
  \det
  \begin{pmatrix}
    w_+ & e_+ \\
    w_- & e_-
  \end{pmatrix}
  \, &= \, -(2\pi i)^3 \, \tfrac{5}{2} \\
  \det
  \begin{pmatrix}
    w_+ & a_+ \\
    w_- & a_-
  \end{pmatrix}
  \, &= \, -(2\pi i)^2 \, \sqrt{5} \, b \\
  \det
  \begin{pmatrix}
    e_+ & a_+ \\
    e_- & a_-
  \end{pmatrix}
  \, &= \, -(2\pi i)^3 \, \tfrac{9}{4} -(2\pi i)^2 \, \sqrt{5} \, d \, .
\end{align*} 
It is straightforward to compute the mixed periods numerically and one obtains
\begin{center}
    \begin{tabular}{l c D{.}{.}{4.50}}
      $w_+$  & = &   320.871302959778116770497485624017226038 \cdots   \\
      $w_-$  & = & -1536.675109826085372724756354590337175648 \cdots i \\
      $e_+$  & = &    -6.893856185212988044137977532235735104 \cdots   \\
      $e_-$  & = &    34.947789474177653892854041280741645293 \cdots i \\
      $a_+$  & = &    37.397710905400938350547117646682006554 \cdots   \\
      $a_-$  & = &  -252.169016964624605484461069839609176011 \cdots i \\
      $b$    & = &  -265.593780202397705806104094596997598070 \cdots i \\
      $d$    & = &    -1.434849336934471921847071711478709892 \cdots i \\
      $c$    & = &     6.128728877854787485401183630654047566 \cdots i \, .
    \end{tabular}
\end{center}

\section{Fibering out}
\label{sec:FiberingOut}

The goal of fibering out is to express the period functions $\Pi$ and their derivatives in terms of integrals of period functions of a family of K3 surfaces. To do this we use the elementary identity 
\begin{equation*}
  \begin{aligned}
    \frac{1}{(2\pi i)^3} \, \Pi_1(z) \, &= \, \sum_{n=0}^\infty \frac{(5 \, n)!}{n!^5} \, z^n \\
    &= \, \frac{1}{2\pi i}\oint \left( \sum_{n=0}^\infty \sum_{k=-n}^\infty \frac{(5\, n+k)!}{n!^4 \, (n+k)!} \, t^k \, z^n \right) \frac{\dd t}{t} \\
    &= \, \frac{1}{2\pi i}\oint \frac{1}{1-t} \, \left(\sum_{n=0}^\infty \frac{(4\, n)!}{n!^4} \, \left(\frac{z}{t \, (1-t)^4}\right)^n  \right) \frac{\dd t}{t} \\
    &= \, \frac{1}{2\pi i}\oint \frac{1}{1-t} \ \pFq{3}{2}{\frac{1}{2},\frac{1}{4},\frac{3}{4}}{1,1}{2^8 \frac{z}{t \, (1-t)^4}} \frac{\dd t}{t} \, ,
  \end{aligned}
\end{equation*}
where the contour can be chosen to be the counterclockwise circle of radius $|t| = 1/5$. Identities like this have previously been studied in \cite{DoranIdentities}. In the remainder of this section we consider analytic continuations and derivatives of this identity, which allow to derive similar identities for other components of $\Pi$ and their derivatives.

\subsection*{Identities for other components of $\boldsymbol{\Pi}$}

To write our identity in a more conceptual form, we consider the operator $\Theta^3 - 2^8 \, t \, (\Theta+1/2) \, (\Theta+1/4) \, (\Theta+3/4)$ with $\Theta= t \dv{}{t}$. This annihilates the rank three hypergeometric function above and for $0 < t < 1/2^8$ a basis of solutions is given by
\begin{align*}
  \varrho(t) \, = \,
  \begin{pmatrix}
    (2\pi i)^2 & 0 & 0 \\
    0 & 2 \pi i & 0 \\
    0 & 0 & 2 
  \end{pmatrix}
  \left(
  \begin{array}{r}
    f_1(t) \\
    f_1(t) \, \log(t) + f_2(t) \\
    \frac{1}{2}f_1(t) \log(t)^2 + f_2(t) \, \log(t) + f_3(t)
  \end{array}
  \right)
\end{align*}
with power series normalized by $f_1(0) = 1$ and $f_2(0) = f_3(0) = 0$. In terms of $\phi_z(t) = \frac{z}{t \, (1-t)^4}$ and $\mathcal{M}_z = \mathbb{P}^1 \smallsetminus \phi_z^{-1}(\{0,1/2^8,\infty \})$ we can then consider the integration map
\begin{align*}
  I_z : \ \pi_1(\mathcal{M}_z, t_0) \, &\rightarrow \mathbb{C}^3 \\
  \gamma \, &\mapsto \int_{\gamma} \frac{1}{1-t} \, (\phi_z^*\varrho)(t) \, \frac{\dd t}{t} \, ,
\end{align*}
where the integrand is understood to be analytically continued along the contour. In terms of this map our original identity reads
\begin{align*}
  \Pi_1(z) \, = \,
  \begin{pmatrix}
    1 \\
    0 \\
    0
  \end{pmatrix}
  \cdot I_z(\gamma_4^{-1}\gamma_6^{-1}\gamma_2^{-1})
\end{align*}
with the following chosen basis of $\pi_1(\mathcal{M}_z, t_0)$ for $t_0 \gg 0$:
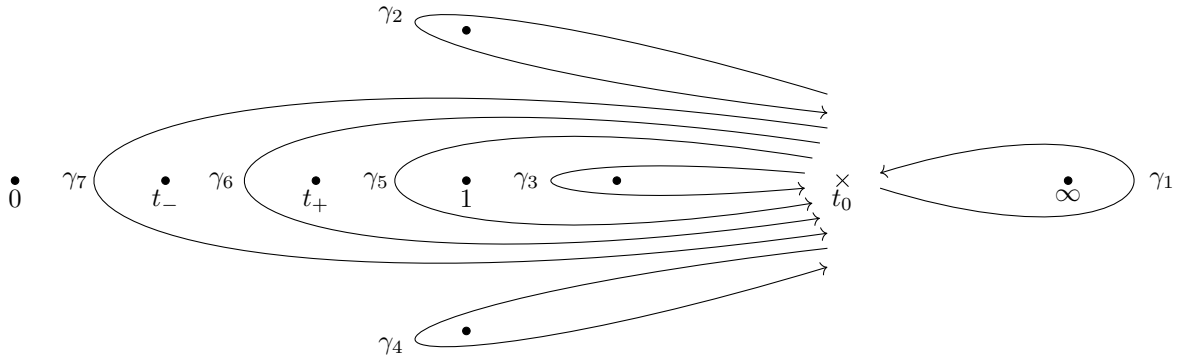
\begin{figure}[h]
  \centering
  \begin{tikzpicture}
    \draw [fill=black] (-2,0) circle (0.05) node [below] {$0$};
    \draw [fill=black] (0,0) circle (0.05) node [below] {$t_-$};
    \draw [fill=black] (2,0) circle (0.05) node [below] {$t_+$};
    \draw [fill=black] (4,0) circle (0.05) node [below] {$1$};
    \draw [fill=black] (6,0) circle (0.05);
    \draw [fill=black] (4,2) circle (0.05);
    \draw [fill=black] (4,-2) circle (0.05);
    \draw [fill=black] (12,0) circle (0.05) node [below] {$\infty$};
    
    \draw (9-0.08,-0.08) -- (9+0.08,0.08);
    \draw (9-0.08,0.08) -- (9+0.08,-0.08);
    \draw (9,0) node [below] {$t_0$};

    \draw [->] (9.5,-0.1) .. controls (14,-1.5) and (14,1.5) .. (9.5,0.1);
    \draw (13.25,0) node {$\gamma_1$};

    \draw [->] (8.5,0.1) .. controls (4,0.5) and (4,-0.5) .. (8.5,-0.1);
    \draw (4.8,0) node {$\gamma_3$};
    
    \draw [->] (8.6,0.3) .. controls (1.2,1.5) and (1.2,-1.5) .. (8.6,-0.3);
    \draw (2.8,0) node {$\gamma_5$};
    
    \draw [->] (8.7,0.5) .. controls (-1.5,2) and (-1.5,-2) .. (8.7,-0.5);
    \draw (0.75,0) node {$\gamma_6$};

    \draw [->] (8.8,0.7) .. controls (-4.2,2.5) and (-4.2,-2.5) .. (8.8,-0.7);
    \draw (-1.2,0) node {$\gamma_7$};
     
    \draw [->] (8.8,1.15) .. controls (2,3.2) and (1,1.8) .. (8.8,0.9);
    \draw (3,2.2) node {$\gamma_2$};

    \draw [<-] (8.8,-1.15) .. controls (2,-3.2) and (1,-1.8) .. (8.8,-0.9);
    \draw (3,-2.2) node {$\gamma_4$};
  \end{tikzpicture}
\end{figure}

Note that $I$ satisfies the cocycle property $I_z(\gamma \, \gamma') = I_z(\gamma) + M_{\gamma}I_z(\gamma')$ with the monodromy matrix $M_{\gamma}$ corresponding to the action on $\phi_z^*\varrho$. In terms of
\begin{align*}
  T \, = \,
  \begin{pmatrix}
    1 & 0 & 0 \\
    1 & 1 & 0 \\
    1 & 2 & 1
  \end{pmatrix}
  \qquad \text{and} \qquad 
  W \, = \, 
  \begin{pmatrix}
    0 & 0 & 2 \\
    0 & -1 & 0 \\
    1/2 & 0 & 0
  \end{pmatrix}
\end{align*}
the monodromy matrices are given as follows:
\begin{align*}
  \begin{array}[h]{c|ccccccc}
    \gamma & \gamma_1 & \gamma_2 & \gamma_3 & \gamma_4 & \gamma_5 & \gamma_6 & \gamma_7 \\ \hline
    M_{\gamma} & T^5 & -T^{-1}WT & -W & -TWT^{-1} & -W & T^{-1}WT^{-2}WT^{-1} & -W
  \end{array}
\end{align*}

To compute analytic continuations of $I$, we need to study how $\mathcal{M}_z$ changes with $z$. For a loop which starts at some $0<z<1/5^5$ and encircles $0$ (respectively $1/5^5$) counterclockwise, the action on the holes of $\mathcal{M}_z$ is depicted by the solid (respectively dashed) arrows below:
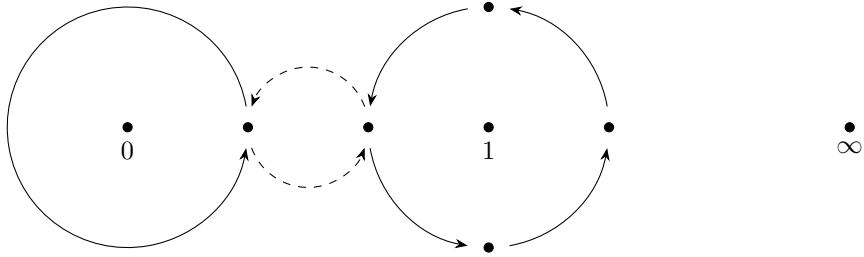
\begin{figure}[h]
  \centering
  \begin{tikzpicture}[scale=0.8]
    \clip (-4.25,-2.1) rectangle(10.5,2.1);
    
    \draw[fill=black] (-2,0) circle (0.5ex) node[below=2] {$0$};
    \draw[fill=black] (0,0) circle (0.5ex);
    \draw[fill=black] (2,0) circle (0.5ex);
    \draw[fill=black] (4,0) circle (0.5ex) node[below=2] {$1$};
    \draw[fill=black] (6,0) circle (0.5ex);
    \draw[fill=black] (4,2) circle (0.5ex);
    \draw[fill=black] (4,-2) circle (0.5ex);
    \draw[fill=black] (10,0) circle (0.5ex) node[below=2] {$\infty$};
    
    \draw[-{Stealth[scale=1]}] (-0.03,0.347) arc (10:350:2);
    \draw[-{Stealth[scale=1]}] (5.97,0.347) arc (10:80:2);
    \draw[-{Stealth[scale=1]}] (3.653,1.97) arc (100:170:2);
    \draw[-{Stealth[scale=1]}] (2.03,-0.347) arc (190:260:2);
    \draw[-{Stealth[scale=1]}] (4.347,-1.97) arc (280:350:2);
    
    \draw[-{Stealth[scale=1]},dashed] (2*1.97-2,2*0.171) arc (20:160:1);
    \draw[-{Stealth[scale=1]},dashed] (2*1.03-2,-2*0.171) arc (200:340:1);
  \end{tikzpicture}
\end{figure}

It follows that e.g.
\begin{equation}
\begin{aligned}
  &(-3 \, \Pi_2 + \Pi_3 - \Pi_4)(z) \\
  = \, & (4\, \Pi+M_{1/5^5} \, M_0 \, \Pi)_1(z) \\
  = \, &
         \begin{pmatrix}
           9 \\
           -8 \\
           2
         \end{pmatrix}
         \cdot I_z(\gamma_1) +
         \begin{pmatrix}
           14 \\
           12 \\
           2
         \end{pmatrix}
         \cdot I_z(\gamma_2) +
         \begin{pmatrix}
           5 \\
           -8 \\
           2
         \end{pmatrix}
         \cdot I_z(\gamma_4) +
         \begin{pmatrix}
           -3 \\
           -12 \\
           -10
         \end{pmatrix}
         \cdot I_z(\gamma_6) +
         \begin{pmatrix}
           2 \\
           4 \\
           2
         \end{pmatrix}
         \cdot I_z(\gamma_7) \\
  = \, &
         \begin{pmatrix}
           9 \\
           -8 \\
           2
         \end{pmatrix}
         \cdot I_z(\gamma_1) +
         \begin{pmatrix}
           14 \\
           12 \\
           2
         \end{pmatrix}
         \cdot I_z(\gamma_2) +
         \begin{pmatrix}
           5 \\
           -8 \\
           2
         \end{pmatrix}
         \cdot I_z(\gamma_4) +
         \begin{pmatrix}
           -1 \\
           -8 \\
           -8
         \end{pmatrix}
         \cdot I_z(\gamma_7) \, ,
\end{aligned}
\label{eq:RegularPeriodIdentity}
\end{equation}

where in the last step we used that $(-3,-12,-10) \cdot I_z(\gamma_6) = (-3,-12,-10) \cdot I_z(\gamma_7)$ since the integrand is holomorphic around $t_\pm$. Avoiding contributions from $I_z(\gamma_6)$ will be beneficial since for $z \rightarrow 1/5^5$ the cycle $\gamma_6$ gets pinched at $t_{\pm} = 1/5$. Also note that, while individual terms $I_z(\gamma)$ can depend on $t_0$, the linear combinations above are independent of $t_0$. More generally, a linear combination $\sum_i v_i \cdot I_z(\gamma_i)$ is independent of $t_0$ as long as $\sum_i v_i \cdot (M_{\gamma_i}-1) = 0$. 

\subsection*{Identities for derivatives of $\boldsymbol{\Pi}$} By taking derivatives and integrating by parts, we can obtain identities for derivatives of $\Pi$. More precisely, for linear combinations $\sum_i v_i \cdot I_z(\gamma_i)$ that satisfy $\sum_i v_i \cdot (M_{\gamma_i}-1) = 0$ one finds that
\begin{align*}
    \left(\dv{}{z}\right)^k \, \sum_i v_i \cdot I_z(\gamma_i) \, = \, \sum_i v_i \cdot I_z^{(k)}(\gamma_i)
\end{align*}
where e.g.
\begin{align*}
  I_z'(\gamma) \, &= \, \frac{1}{z} \int_{\gamma} \frac{5}{(1-5 \, t)^2} \, (\phi_z^*\varrho)(t) \, \dd t \\
  I_z''(\gamma) \, &= \, \frac{1}{z^2} \int_{\gamma} \frac{30 \, t \, (3-5 \, t)}{(1-5 \, t)^4} \, (\phi_z^*\varrho)(t) \, \dd t \, .
\end{align*}
Note that for $z < 1/5^5$ the integrands above have poles at $t=1/5$ but the residues vanish. 

\newpage

\subsection*{Evaluation at $\boldsymbol{z=1/5^5}$}
Because of the cocycle property, $I_z$ is uniquely determined by its values on generators of $\pi_1(\mathcal{M}_z, t_0)$. Naively, this leaves $7 \cdot 3 = 21$ complex degrees of freedom, but there are some simplifications we can use:
\begin{itemize}
\item[-] The integrand is holomorphic at $t=1$ and hence $I_z(\gamma_5) = I_z(\gamma_3)$.
\item[-] For $z = 1/5^5$ the points $t_\pm$ are equal and we don't have to consider $\gamma_6$ anymore. Also, the integrand is holomorphic around $t_{\pm} = 1/5$ and hence $I_{1/5^5}(\gamma_7) = I_{1/5^5}(\gamma_3)$.
\item[-] In the limit $t_0 \rightarrow \infty$ we have $I_z(\gamma_1) = 0$.
\item[-] There are two-dimensional spaces of integrands which are holomorphic around the points encircled by $\gamma_2$, $\gamma_3$ and $\gamma_4$, respectively. It follows that $I_z(\gamma_2) \sim (2,-2,3)$, $I_z(\gamma_3) \sim (2,0,1)$ and $I_z(\gamma_4) \sim (2,2,3)$.
\item[-] For real $t_0$ we must have $\overline{I_z(\gamma_3)} = \mathrm{diag}(1,-1,1) \cdot I_z(\gamma_3)$ and $\overline{I_z(\gamma_2)} = \mathrm{diag}(1,-1,1) \cdot I_z(\gamma_4^{-4})$.
\end{itemize}
We conclude that for $z = 1/5^5$ and $t_0 \rightarrow \infty$ there are only three real degrees of freedom. We capture these by defining\footnote{That the image lies in $\mathbb{Z}^3$ follows from the given first entries being even and the fact that for every $M \in \langle T, W\rangle$ either $M$ or $M \, W$ has integer entries.} three cocycles $\widetilde{r}_\pm,\widetilde{r}_{\text{b}} : \pi_1(\mathcal{M}_{1/5^5}, t_0) \rightarrow \mathbb{Z}^3$ which vanish on $\gamma_1$, satisfy $\widetilde{r}(\gamma_7) = \widetilde{r}(\gamma_5) = \widetilde{r}(\gamma_3)$ and are otherwise given as follows:
\begin{table}[h]
  \centering
  \tabulinesep=0.5mm
  \begin{tabu}[h]{c|ccc}
    $\gamma$ & $\gamma_2$ & $\gamma_3$ & $\gamma_4$ \\ \hline
    $\widetilde{r}_+(\gamma)$ & $\begin{pmatrix} -2 \\ 2 \\ -3 \end{pmatrix}$ & $\begin{pmatrix} -4 \\ 0 \\ -2 \end{pmatrix}$ & $\begin{pmatrix} -2 \\ -2 \\ -3 \end{pmatrix}$ \\
    $\widetilde{r}_-(\gamma)$ & $\begin{pmatrix} -2 \\ 2 \\ -3 \end{pmatrix}$ & $\begin{pmatrix} 0 \\ 0 \\ 0 \end{pmatrix}$ & $\begin{pmatrix} 2 \\ 2 \\ 3 \end{pmatrix}$ \\
    $\widetilde{r}_{\text{b}}(\gamma)$ & $\begin{pmatrix} -2 \\ 2 \\ -3 \end{pmatrix}$ & $\begin{pmatrix} -2 \\ 0 \\ -1 \end{pmatrix}$ & $\begin{pmatrix} -2 \\ -2 \\ -3 \end{pmatrix}$ \\
  \end{tabu}
\end{table}

Note that
\begin{align*}
  \widetilde{r}_{\text{b}}(\gamma) \, = \, (M_\gamma - 1)
  \begin{pmatrix}
    0 \\
    0 \\
    1
  \end{pmatrix}
\end{align*}
and hence $\widetilde{r}_b$ is a coboundary (up to a scaling the unique coboundary that vanishes on $\gamma_1$). \\

The simplifications above apply analogously for the derivatives $I^{(k)}_{1/5^5}$, except that in general there are contributions from $I_{1/5^5}^{(k)}(\gamma_7\gamma_5^{-1})$, i.e. from the residue at $t=1/5$. It follows that there are $\omega_+,\eta_+,\alpha_+, \omega_{\mathrm{b}}, \eta_{\mathrm{b}}, \alpha_{\mathrm{b}} \in \mathbb{R}$ and $\omega_-,\eta_-,\alpha_- \in i \mathbb{R}$ such that for all $\gamma \in \pi_1(\mathcal{M}_{1/5^5}, t_0)$
\begin{align}
  \lim_{t_0 \rightarrow \infty}
  \begin{pmatrix}
    1 & 0 & 0 \\
    0 & \frac{7}{10} \frac{1}{5^5} & \frac{1}{2} \frac{1}{5^{10}} \\
    0 & -\frac{1}{5^5} & 0
  \end{pmatrix}
  \begin{pmatrix}
    I_{1/5^5} \\
    I'_{1/5^5} \\
    I''_{1/5^5}
  \end{pmatrix}
  (\gamma) \, \equiv\,
  \begin{pmatrix}
    \omega_+ & \omega_- \\
    \eta_+ & \eta_- \\
    \alpha_+ & \alpha_-
  \end{pmatrix}
  \begin{pmatrix}
    \widetilde{r}_+ \\
    \widetilde{r}_-
  \end{pmatrix}
  (\gamma) +
  \begin{pmatrix}
    \omega_{\text{b}} \\
    \eta_{\text{b}} \\
    \alpha_{\text{b}} 
  \end{pmatrix}
  \widetilde{r}_{\text{b}}(\gamma)
  \ \ \text{mod} \,
  \begin{pmatrix}
    0 \\
    0 \\
    \frac{1}{2} (2\pi i)^2 \sqrt{5} \, \mathbb{Z}^3
  \end{pmatrix}
  \label{eq:FullNonModularPeriodIdentity}
\end{align}
and from Equation (\ref{eq:MixedPeriods}) and Equation (\ref{eq:RegularPeriodIdentity}) we can read off that
\begin{align*}
  \begin{pmatrix}
    \omega_+ & \omega_- \\
    \eta_+ & \eta_- \\
    \alpha_+ & \alpha_-
  \end{pmatrix}
  \, = \,
  \begin{pmatrix}
    w_+ & w_- \\
    e_+ & e_- \\
    a_+ & a_-
  \end{pmatrix}
  \,
  \begin{pmatrix}
    -1/4 & 0 \\
    0 & -1/10
  \end{pmatrix}
  \, .
\end{align*}
Here, we used that the residues of the integrand in the second row vanish while for the last row we have $-\frac{1}{5^5} I'_{1/5^5}(\gamma_7\gamma_5^{-1})  =  (2\pi i)^2 \, \sqrt{5} \, (-3,2,-3/2)$. 

\newpage

\section{Modularity of the mixed periods}
\label{sec:Modularity}

We now prove the modularity of the mixed periods using the results form the previous section and the modularity of $\phi_{1/5^5}^*\varrho$. First note that in terms of the normalized Hauptmodul $h_2(\tau) = q^{-1} + O(q)$ of $\Gamma_0^*(2)$ and $t_2 = \frac{1}{h_2+104}$, we have
\begin{align*}
  t_2^* \varrho(\tau) \, = \, (2\pi i)^2 \, 
  \begin{pmatrix}
    1 \\
    \tau \\
    \tau^2
  \end{pmatrix} \, E(\tau)
\end{align*}
with the unique Eisenstein series $E \in M_2(\Gamma_0(2))$ normalized by $E(\tau) = 1+ O(q)$. It is not clear that this helps since we need to evaluate the pullback $\phi_z^*\varrho$ and not $\varrho$. However, for $z = 1/5^5$, we can use that there are modular solutions to $\phi_{1/5^5}(t(\tau)) = t_2(5 \, \tau)$. We fix the solution that is given in terms of the normalized Hauptmodul $h_{50}(\tau) = q^{-1} + O(q)$ of $\Gamma_0^*(50)$ by $t_{50}  =  \frac{(1-h_{50}) \, (3+h_{50})^2}{5 \, (1-h_{50} - h_{50}^2)}$. In terms of the Dedekind eta function, the Hauptmodul $h_{50}$ can be expressed by
\begin{align*}
  h_{50}(\tau) \, = \, \frac{\eta(\tau) \, \eta(50 \, \tau)}{\eta(2 \, \tau) \, \eta(25 \, \tau)} + \frac{\eta(2 \, \tau) \, \eta(25 \, \tau)}{\eta(\tau) \, \eta(50 \, \tau)} - 1 \, .
\end{align*}
We obtain the pullback
\begin{align*}
  t_{50}^* \left(\frac{1}{t \, (1-t)} \ \phi_{1/5^5}^*\varrho \, \dd t \right) = \,
  (2\pi i)^3 \, 
  \begin{pmatrix}
    1 \\
    5 \, \tau \\
    25 \, \tau^2
  \end{pmatrix}
  \,
  f_{50} \, \dd \tau \, ,
\end{align*}
where $f_{50}(\tau) = 5 \, f(\tau) - 20 \, f(2 \, \tau)$ in terms of the newform $f$ associated with the conifold fiber. The action of $\Gamma_0^*(50)$ on $f_{50}$ gives $f_{50}|_4 \gamma = \chi_{\gamma} \, f_{50}$ where the character $\chi$ is trivial on $\Gamma_0(50)$ and evaluates to $-1$ and $1$ for the Atkin-Lehner transformations $W_2$ and $W_{25}$, respectively. For the evaluation of derivatives of the period functions we also define the modular forms
\begin{align*}
  g_{50} \, = \, -\frac{5 \, t_{50} \, (1-t_{50})}{(1-5 \, t_{50})^2} \, f_{50} \, , \quad 
  F_{50} \, = \, \frac{t_{50}\, (1-t_{50}) \, (7+20 \, t_{50}+25 \, t_{50}^2)}{2 \, (1-5 \, t_{50})^4} \, f_{50} \, .
\end{align*}

\begin{theorem}
  The numbers $w_\pm$, $e_\pm$ and $a_\pm$ are periods of the modular forms $f_{50}$, $F_{50}$ and $g_{50}$. More precisely, there are cocycles $r_\pm : \Gamma_0^*(50) \rightarrow \mathbb{Z}^3$ such that for every $\gamma = \left(\substack{ a\, b\\ c\, d}\right) \in \Gamma_0^*(50)$ and every choice of contour we have
  \begin{align*}
    (2\pi i)^3 \int_{\infty}^{\gamma \infty}
    \begin{pmatrix}
      1 \\
      5 \, \tau \\
      25 \, \tau^2
    \end{pmatrix}
    \,
    (f_{50}, F_{50}, g_{50}) \, \dd \tau \, \equiv \, &
    (r_+, r_-)(\gamma) \,
    \begin{pmatrix}
      \omega_+ & \eta_+ & \alpha_+ \\
      \omega_- & \eta_- & \alpha_- 
    \end{pmatrix}
    \\ 
    & +
    \left(
    \chi_{\gamma}
    \begin{pmatrix}[1.25]
      \frac{1}{25}c^2 \\
      \frac{1}{5}a \, c \\
      a^2
    \end{pmatrix}
    -
    \begin{pmatrix}[1.25]
      0 \\
      0 \\
      1
    \end{pmatrix}
      \right) (\omega_{\text{b}}, \eta_{\text{b}}, \alpha_{\text{b}})
    \\
    & \mathrm{mod} \ \left(0,0,\tfrac{1}{2} (2\pi i)^2 \sqrt{5} \, \mathbb{Z}^3\right)
  \end{align*}
  in terms of the re-scaled periods from Equation (\ref{eq:FullNonModularPeriodIdentity}).
\end{theorem}

\begin{proof}
  We consider the largest $\tau_0 \in i \mathbb{R}$ such that $t_{50}(\tau_0) = t_0$. Further, for every $\gamma = \left(\substack{ a\, b\\ c\, d}\right) \in \Gamma_0^*(50)$, we choose a path in $t_{50}^{-1}(\mathcal{M}_{1/5^5})$ from $\tau_0$ to $\gamma \tau_0$ and denote the associated image in $\pi_1(\mathcal{M}_{1/5^5}, t_0)$ by $\widetilde{\gamma}$. This allows us to define the cocycles $r_{\pm}$ by $r_{\pm}(\gamma) = \widetilde{r}_{\pm}(\widetilde{\gamma})$ (which are independent of the choice of path from $\tau_0$ to $\gamma \tau_0$). The monodromy matrix associated with $\widetilde{\gamma}$ is
  \begin{align*}
    M_{\widetilde{\gamma}} \, = \, \chi_{\gamma}
    \begin{pmatrix}[1.25]
      d^2 & \tfrac{2}{5}c \, d & \tfrac{1}{25}c^2 \\
      5 \, b \, d & a \, d + b \, c & \tfrac{1}{5} \, a \, c \\
      25 \, b^2 & 10 \, a \, b & a^2
    \end{pmatrix}
    \, .
  \end{align*}
  Pulling back the integrals in Equation (\ref{eq:FullNonModularPeriodIdentity}) by $t_{50}$ the result follows immediately. Note that, using the action by Hecke operators, one can further show that $\omega_{\text{b}} = \eta_{\text{b}} = 0$ and numerical computations suggest that $\alpha_{\text{b}} \equiv \frac{1}{10} (2\pi i)^2 \, \sqrt{5} \ \text{mod} \ \frac{1}{2} (2\pi i)^2 \sqrt{5} \, \mathbb{Z}$.
\end{proof}

\newpage

\begin{theorem}
  The numbers $b, c, d$ can be expressed in terms of integrals of the modular forms $f_{50}$, $g_{50}$ and $F_{50}$. More precisely, in terms of the CM points $\tau_\pm = \pm \frac{2}{5}+i \, \frac{\sqrt{2}}{10}$, we have
  \begin{align*}
    b \, &= \, (2\pi i)^3 \, \int_{\tau_-}^{\tau_+} f_{50}(\tau) \, \dd \tau \\
    c \, &= \, \lim_{\epsilon \downarrow 0} \left( (2\pi i)^3 \, \int_{\tau_-}^{\tau_+} g_{50}(\tau+i\tfrac{\sqrt{2}}{5} \, \epsilon) \, \dd \tau +2\pi i \, \sqrt{5} \, \left(\tfrac{1}{\epsilon} + \log(-5 \, t_{50}'(\tau_-) \, t_{50}'(\tau_+) \, \epsilon^2) \right)\right) \\
    d \, &= \, \lim_{\epsilon \downarrow 0} \left( (2\pi i)^3 \, \int_{\tau_-}^{\tau_+} F_{50}(\tau+i\tfrac{\sqrt{2}}{5} \, \epsilon) \, \dd \tau +2\pi i \, \sqrt{5} \, \left(-\tfrac{1}{5\, t_{50}'(\tau_-) \, t_{50}'(\tau_+)} \, (\tfrac{1}{\epsilon^3}+1) -\tfrac{257}{480} \right)\right) \, ,
  \end{align*}
  where
  \begin{align*}
    t_{50}'(\tau_{\pm}) \, = \, \mp \frac{\Gamma (\tfrac{1}{8})^2 \Gamma (\tfrac{3}{8})^2}{2 \, \sqrt{10} \, \pi ^2}\, .
  \end{align*}
  Here, the integral in the expression for $c$ is along the straight line between $\tau_-$ and $\tau_+$.
\end{theorem}

\begin{proof}
   We use the identity
  \begin{align*}
    \Pi_1^{(k)}(z) \, &= \, \begin{pmatrix}
    1 \\
    0 \\
    0
  \end{pmatrix}
  \cdot I_z^{(k)}(\gamma_4^{-1}\gamma_6^{-1}\gamma_2^{-1})
  \end{align*}
  from Section \ref{sec:FiberingOut}. In the limit $\delta = 1-5^5 \, z \rightarrow 0$, the loop $\gamma_6$ gets pinched at $1/5$ and we have to be careful with divergencies. We take the limit $t_0 \rightarrow \infty$ and decompose $\gamma_4^{-1}\gamma_6^{-1}\gamma_2^{-1}$ into the three parts $\gamma_{-\epsilon}$, $\gamma_0$, $\gamma_{+\epsilon}$ depicted below:
  
  
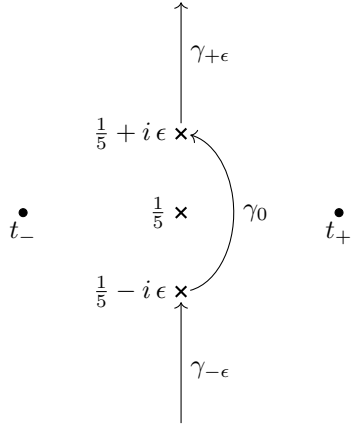
\begin{figure}[h]
    \centering

    \tikzset{
      partial ellipse/.style args={#1:#2:#3}{
        insert path={+ (#1:#3) arc (#1:#2:#3)}
      }
    }
    
    \begin{tikzpicture}[scale=0.7]
      
      \draw[fill=black] (-3,0) circle (0.5ex) node [below] {$t_-$};
      \draw[fill=black] (3,0) circle (0.5ex) node [below] {$t_+$};
      
      \draw[thick] (-0.1,-0.1) -- (0.1,0.1);
      \draw[thick] (-0.1,0.1) -- (0.1,-0.1);
      \draw (-0.1,0) node [left] {$\tfrac{1}{5}$};

      \begin{scope}[shift={(0,1.5)}]
        \draw[thick] (-0.1,-0.1) -- (0.1,0.1);
        \draw[thick] (-0.1,0.1) -- (0.1,-0.1);
        \draw (-0.1,0) node [left] {$\tfrac{1}{5}+i \, \epsilon$};
      \end{scope}
      
      \begin{scope}[shift={(0,-1.5)}]
        \draw[thick] (-0.1,-0.1) -- (0.1,0.1);
        \draw[thick] (-0.1,0.1) -- (0.1,-0.1);
        \draw (-0.1,0) node [left] {$\tfrac{1}{5}-i \, \epsilon$};
      \end{scope}
      
      \draw[->] (0,0) [partial ellipse=-80:80:1cm and 1.5cm];
      \draw (1,0) node [right] {$\gamma_0$};
      
      \draw[->] (0,-4) -- (0,-1.7);
      \draw (0,-3) node [right] {$\gamma_{-\epsilon}$};

      \draw[->] (0,1.7) -- (0,4);
      \draw (0,3) node [right] {$\gamma_{+\epsilon}$};
    \end{tikzpicture}
  \end{figure}

  \FloatBarrier
  
  \noindent For fixed $\epsilon$, we can take the limit $\delta \rightarrow 0$ for the integrals over $\gamma_{\pm\epsilon}$ and pull these back to the upper half-plane. The asymptotics of the integral over $\gamma_0$ in the limit where we first take $\delta \rightarrow 0$ and then~$\epsilon \rightarrow 0$ can be obtained from the expansion
  \begin{align*}
    \varrho(t) \, = \,
    \begin{pmatrix}
      -\frac{1}{2\pi} \, \Gamma(1/8)^2 \, \Gamma(3/8)^2 & 4\sqrt{2}\, \pi & -\frac{2}{\pi} \, \Gamma(5/8)^2 \, \Gamma(7/8)^2 \\[6pt]
      -\frac{\sqrt{2} i}{4\pi} \, \Gamma(1/8)^2 \, \Gamma(3/8)^2 & 0 & \frac{\sqrt{2}i}{\pi} \, \Gamma(5/8)^2 \, \Gamma(7/8)^2 \\[6pt]
      \frac{1}{4\pi} \, \Gamma(1/8)^2 \, \Gamma(3/8)^2 & 2\sqrt{2} \, \pi & \frac{1}{\pi} \, \Gamma(5/8)^2 \, \Gamma(7/8)^2
    \end{pmatrix}
    \begin{pmatrix}
      1+\frac{3}{16} x + O(x^2) \\[6pt]
      \sqrt{x} \, (1+ O(x)) \\[6pt]
      x + O(x^2)
    \end{pmatrix}
  \end{align*}
  in terms of $x = 1-2^8 \, t$. From this expansion and the monodromy matrices, we also find that for~$\epsilon \rightarrow 0$ we can pull back the sum of the paths $\gamma_{\pm\epsilon}$ to the path from $\tau_- = -\frac{2}{5}+i \, \frac{\sqrt{2}}{10}$ to $\tau_+ = \frac{2}{5}+i \, \frac{\sqrt{2}}{10}$. The evaluation of $b$ is now straightforward since $\gamma_0$ does not contribute in the limit $\epsilon,\delta \rightarrow 0$ and thus
  \begin{align*}
    b \, = \, \lim_{z \rightarrow 1/5^5} \, \Pi_1(z) \, = \, (2\pi i)^3 \, \int_{\tau_-}^{\tau_+} f_{50}(\tau) \, \dd \tau \, .
  \end{align*}
  To obtain $d$ and $c$, we need to subtract divergent contributions and include the contributions of~$\gamma_0$. After reparametrizing $\epsilon$, this gives
  \begin{align*}
    c \, &= \, \lim_{z \rightarrow 1/5^5} \, \left( \Pi_1(z) +2\pi i \, \sqrt{5} \, \log(\delta) \right) \\
         &= \,  \lim_{\epsilon \downarrow 0} \left( (2\pi i)^3 \, \int_{\tau_-}^{\tau_+} g_{50}(\tau+i\tfrac{\sqrt{2}}{5} \, \epsilon) \, \dd \tau +2\pi i \, \sqrt{5} \, \left(\tfrac{1}{\epsilon} + \log(-5 \, t_{50}'(\tau_-) \, t_{50}'(\tau_+) \, \epsilon^2) \right)\right) \\
    d \, &= \, \lim_{z \rightarrow 1/5^5} \, \left(\frac{1}{2} \dv{^2}{\delta^2} \Pi_1(z)-\frac{7}{10} \dv{}{\delta} \Pi_1(z)+2\pi i \, \sqrt{5} \, \left(\frac{1}{2 \, \delta}+\frac{7}{20} \right) \right) \\
         &= \, \lim_{\epsilon \downarrow 0} \left( (2\pi i)^3 \, \int_{\tau_-}^{\tau_+} F_{50}(\tau+i\tfrac{\sqrt{2}}{5} \, \epsilon) \, \dd \tau +2\pi i \, \sqrt{5} \, \left(-\tfrac{1}{5\, t_{50}'(\tau_-) \, t_{50}'(\tau_+)} \, (\tfrac{1}{\epsilon^3}+1) -\tfrac{257}{480} \right)\right)  \\
  \end{align*}
  and from the expansion of $\varrho$ one obtains the given values of $t_{50}'(\tau_{\pm})$. Note that due to the residues of $g_{50}$, variations of the contour of integration can shift the integral in the expression for $c$ by integer multiples of $(2\pi i)^2 \, \sqrt{5}$. The equality above holds for the straight line between $\tau_-$ and $\tau_+$.
\end{proof}

\section{Outlook}
The method outlined in this paper allows to prove the modularity of mixed period matrices of many families of Calabi-Yau threefolds. However, it also leads to open questions and possible future applications, some of which we want to mention here. 

\subsection*{Periods of meromorphic modular forms with non-vanishing residues}
The modular form $g_{50}$ has associated periods which are well-defined modulo contributions from its residues. In particular, $g_{50}$ must be a Hecke eigenform modulo modular forms whose periods lie in $\frac{1}{2}(2\pi i)^2 \sqrt{5} \, \mathbb{Z}^3$. In \cite{MagneticModularForms} it has been conjectured that such boundary terms come from magnetic modular forms (coined in \cite{BroadhurstZudilin}) and for our example numerical computations indeed suggest that for all primes $p \neq 2, 5$ the Hecke action gives $g_{50}|_4T_p \equiv a_p \, g_{50}$ modulo magnetic modular forms (i.e. modular forms whose $n$-th Fourier coefficients are divisible by $n$). 

\subsection*{Explicit algebraic correspondences}
While we have discussed the ``fibering out'' purely on the level of periods, it is not too hard to turn it into an explicit algebraic correspondence between the conifold fiber and a Kuga-Sato threefold. Such correspondences might be useful for future studies and we plan on giving more details in a future publication. 

\subsection*{Height and leading coefficients of $\boldsymbol{L}$-functions}
In \cite{KilianThesis} the relation
\begin{align*}
    1+ \frac{1}{2\pi i \, \sqrt{5} \, w_-} \, \det
  \left(
  \begin{array}{ll}
    b & c \\
    w_- & a_-
  \end{array}
  \right)
 \, = \, 
 -\frac{5}{3} \log 5 - \frac{125}{6} \frac{2 \pi i \, \sqrt{5} \, L'(f\otimes \chi,2)}{w_-}
\end{align*}
between a height and a leading coefficient of an $L$-function has been conjectured based on numerical computations. Here, $\chi$ is the quadratic Dirichlet character associated with $\mathbb{Q}(\sqrt{5})$. The identification of all occurring numbers in terms of integrals of modular forms might be a first step towards proving this relation. 

\newpage

\appendix

\section{Results for other families of Calabi-Yau threefolds}
\label{sec:AppendixOtherFamilies}

For analogous proofs for other hypergeometric families of Calabi-Yau threefolds, one can use the  identities given in Table 12 in \cite{DoranIdentities}. Below we reproduce a part of this table for our use case. Each row in the table contains hypergeometric indices $\{a_1,a_2,a_3,a_4\}$, hypergeometric indices $\{b_1,b_2,b_3\}$ and parameters $k$, $l$, $\beta$ and corresponds to the identity
\begin{align*}
  \pFq{4}{3}{a_1,a_2,a_3,a_4}{1,1,1}{z} \, = \, \frac{1}{2\pi i}\oint_{|t| = \frac{k}{k+l}} \frac{1}{(1-t)^\beta} \ \pFq{3}{2}{b_1,b_2,b_3}{1,1}{\frac{k^k \, l^l}{(k+l)^{k+l}} \frac{z}{t^k \, (1-t)^l}} \frac{\dd t}{t}
\end{align*}
for $|z| \leq 1$:

\begin{table}[h]
  \centering
  \definecolor{light-gray}{gray}{0.5}
  \begin{align*}
    \renewcommand{\arraystretch}{1.2}
    \begin{array}[t]{c !{\color{black}\vrule} c !{\color{black}\vrule} c !{\color{black}\vrule} c}
      a_1,a_2,a_3,a_4 & b_1,b_2,b_3 & k,l,\beta & N \\ \hline
      \frac{1}{2},\frac{1}{2},\frac{1}{2},\frac{1}{2} & \frac{1}{2},\frac{1}{2},\frac{1}{2}  & 1,1,1  & 16  \\ \arrayrulecolor{light-gray} \hline
      \frac{1}{4},\frac{1}{3},\frac{2}{3},\frac{3}{4} & \frac{1}{3},\frac{1}{2},\frac{2}{3}  & 2,2,1  & 48  \\
                      & \frac{1}{3},\frac{1}{2},\frac{2}{3}  & 1,1,\frac{1}{2}  & 48  \\
                      & \frac{1}{4},\frac{1}{2},\frac{3}{4}  & 2,1,1  & 18  \\  \hline
      \frac{1}{4},\frac{1}{2},\frac{1}{2},\frac{3}{4} & \frac{1}{2},\frac{1}{2},\frac{1}{2}  & 2,2,1  & 64  \\
                      & \frac{1}{2},\frac{1}{2},\frac{1}{2}  & 1,1,\frac{1}{2}  & 64  \\
                      & \frac{1}{3},\frac{1}{2},\frac{2}{3}  & 3,1,1  & 48  \\
                      & \frac{1}{4},\frac{1}{2},\frac{3}{4}  & 1,1,1  & 8  \\  \hline
      \frac{1}{5},\frac{2}{5},\frac{3}{5},\frac{4}{5} & \frac{1}{3},\frac{1}{2},\frac{2}{3}  & 3,2,1  & 75  \\
                      & \frac{1}{4},\frac{1}{2},\frac{3}{4}  & 4,1,1  & 50  \\ \hline
      \frac{1}{3},\frac{1}{3},\frac{2}{3},\frac{2}{3} & \frac{1}{3},\frac{1}{2},\frac{2}{3}  & 2,1,1  & 27  \\ \hline
      \frac{1}{4},\frac{1}{4},\frac{3}{4},\frac{3}{4} & \frac{1}{4},\frac{1}{2},\frac{3}{4}  & 2,2,1  & 32  \\
                      & \frac{1}{4},\frac{1}{2},\frac{3}{4}  & 1,1,\frac{1}{2}  & 32  \\ \hline
      \frac{1}{3},\frac{1}{2},\frac{1}{2},\frac{2}{3} & \frac{1}{2},\frac{1}{2},\frac{1}{2}  & 2,1,1  &  \text{---}  \\
                      & \frac{1}{3},\frac{1}{2},\frac{2}{3}  & 1,1,1  & 12  \\ \hline
      \frac{1}{6},\frac{1}{2},\frac{1}{2},\frac{5}{6} & \frac{1}{2},\frac{1}{2},\frac{1}{2}  & 2,1,\frac{1}{2}  & \text{---}  \\
                      & \frac{1}{3},\frac{1}{2},\frac{2}{3}  & 3,3,1  & \text{---}  \\
    \end{array}
    \qquad
    \begin{array}[t]{c !{\color{black}\vrule} c !{\color{black}\vrule} c !{\color{black}\vrule} c}
      a_1,a_2,a_3,a_4 & b_1,b_2,b_3 & k,l,\beta & N \\ \hline
                      & \frac{1}{4},\frac{1}{2},\frac{3}{4}  & 1,2,\frac{1}{2}  & 72  \\
                      & \frac{1}{6},\frac{1}{2},\frac{5}{6}  & 1,1,1  & \text{---}   \\ \arrayrulecolor{light-gray} \hline
      \frac{1}{6},\frac{1}{3},\frac{2}{3},\frac{5}{6} & \frac{1}{3},\frac{1}{2},\frac{2}{3}  & 2,1,\frac{1}{2}  & 108  \\
                      & \frac{1}{4},\frac{1}{2},\frac{3}{4}  & 4,2,1  & \text{---}   \\
                      & \frac{1}{6},\frac{1}{2},\frac{5}{6}  & 2,1,1  & \text{---}  \\ \hline
      \frac{1}{8},\frac{3}{8},\frac{5}{8},\frac{7}{8} & \frac{1}{3},\frac{1}{2},\frac{2}{3}  & 3,1,\frac{1}{2}  & \text{---}   \\
                      & \frac{1}{4},\frac{1}{2},\frac{3}{4}  & 4,4,1  & 128  \\
                      & \frac{1}{4},\frac{1}{2},\frac{3}{4}  & 2,2,\frac{1}{2}  & 128  \\
                      & \frac{1}{6},\frac{1}{2},\frac{5}{6}  & 1,3,\frac{1}{2}  & \text{---}   \\ \hline
      \frac{1}{6},\frac{1}{4},\frac{3}{4},\frac{5}{6} & \frac{1}{4},\frac{1}{2},\frac{3}{4}  & 2,1,\frac{1}{2}  & 72  \\
                      & \frac{1}{6},\frac{1}{2},\frac{5}{6}  & 2,2,1  & \text{---}   \\
                      & \frac{1}{6},\frac{1}{2},\frac{5}{6}  & 1,1,\frac{1}{2}  & \text{---}   \\ \hline
      \frac{1}{10},\frac{3}{10},\frac{7}{10},\frac{9}{10} & \frac{1}{4},\frac{1}{2},\frac{3}{4}  & 4,1,\frac{1}{2}  & 200  \\
                      & \frac{1}{6},\frac{1}{2},\frac{5}{6}  & 2,3,\frac{1}{2}  & \text{---}   \\ \hline
      \frac{1}{6},\frac{1}{6},\frac{5}{6},\frac{5}{6} & \frac{1}{6},\frac{1}{2},\frac{5}{6}  & 2,1,\frac{1}{2}  & \text{---}   \\ \hline
      \frac{1}{12},\frac{5}{12},\frac{7}{12},\frac{11}{12} & \frac{1}{4},\frac{1}{2},\frac{3}{4}  & 4,2,\frac{1}{2}  &  \text{---}  \\ \hline
    \end{array}
  \end{align*}
\end{table}
\noindent Using these identities, one proceeds as for our main example with hypergeometric indices $\{\frac{1}{5},\frac{2}{5},\frac{3}{5},\frac{4}{5}\}$. However, it is not guaranteed that the pullback of the integrands to the upper half-plane gives modular forms. If this is the case, the last column in the table gives the level $N$ and one can proceed completely analogous to our main example. \\ \, \\
We remark that the identities which do not give cusp forms $f_N$ still give interesting identities. In general, one obtains instead of a cusp form a modular form multiplied by an algebraic function of a modular function. \\

As a non-hypergeometric example, we mention the identity 
\begin{align*}
    f_{l}(z) \, = \, \frac{1}{2\pi i} \oint_{|t| = \frac{1}{l+1}} \frac{1}{(1-z/t) \, (1-t)} \, f_{l-1}\left( \frac{z}{(1-z/t) \, (1-t)} \right) \, \frac{\dd t}{t}
\end{align*}
for
\begin{align*}
    f_l(z) \, = \, \sum_{n_1=\cdots =n_{l+1} = 0}^\infty \left( \frac{(n_1+\cdots +n_{l+1})!}{n_1!\cdots n_{l+1}!} \right)^2 \, z^{n_1+\cdots n_{l+1}} \, ,
\end{align*}
which holds for $|z| \leq 1/(l+1)^2$. Each $f_{l}$ is associated with a family of Calabi-Yau $(l-1)$-folds which in \cite{BananaIntegrals} have also been associated with so-called $l$-loop banana integrals. The family of Calabi-Yau threefolds associated with $f_4$ has been studied in \cite{AttractorPoints} with regards to attractor points. For this, the identity above combined with the analysis exemplified in the main part of this paper allows to prove the modularity of the mixed period matrices associated with the limits $z \rightarrow 1/9$ and $z \rightarrow 1$.

\bibliographystyle{plain}
\bibliography{References}

\vspace{20pt}

\end{document}